\documentclass[12pt]{article}

\usepackage{graphicx}
\usepackage{subcaption}
\usepackage{amsmath}
\usepackage{dsfont}
\usepackage{amssymb}
\usepackage{abstract}
\usepackage[auth-sc,affil-sl]{authblk}
\usepackage{hyperref}
\usepackage{parskip}
\usepackage{enumerate}
\usepackage{fancyhdr}
\usepackage[xcdraw,table]{xcolor}
\usepackage[round]{natbib}
\usepackage{multirow}
\usepackage{algorithm}
\usepackage{algorithmicx}
\usepackage{algpseudocode}
\usepackage[table]{xcolor}
\usepackage{multirow}
\usepackage{amsthm}

\usepackage[margin=1.0in]{geometry}
\usepackage{fancyvrb}

\newcounter{probnum}
\allowdisplaybreaks

\definecolor{tabblue}{rgb}{.870588,.905882,.94902}
\definecolor{gray}{rgb}{0.5,0.5,0.5}
\definecolor{black}{rgb}{0,0,0}
\definecolor{white}{rgb}{1,1,1}
\definecolor{blue}{rgb}{0.0,0.0,1}

\definecolor{green}{rgb}{0,0.5,0}

\definecolor{yellow}{rgb}{1,0.549,0}

\definecolor{red}{rgb}{0.6,0.0,0.0}

\definecolor{darkred}{rgb}{0.9,0.4,0}

\definecolor{purple}{rgb}{0.58,0,0.827}

\definecolor{backgcode}{rgb}{0.97,0.97,0.8}
\definecolor{Brown}{cmyk}{0,0.81,1,0.60}
\definecolor{OliveGreen}{cmyk}{0.64,0,0.95,0.40}
\definecolor{CadetBlue}{cmyk}{0.62,0.57,0.23,0}

\newcommand{\qu}[1]{``{#1}''}
\newcommand{\bv}[1]{\boldsymbol{#1}}

\newcommand{\iid}{~{\buildrel iid \over \sim}~}

\newcommand{\half}{\frac{1}{2}}

\newcommand{\x}{\bv{x}}

\renewcommand{\c}{\bv{c}}

\newcommand{\z}{\bv{z}}

\newcommand{\reals}{\mathbb{R}}

\newcommand{\beqn}{\vspace{-0.25cm}\begin{eqnarray*}}
\newcommand{\eeqn}{\end{eqnarray*}}
\newcommand{\bneqn}{\vspace{-0.25cm}\begin{eqnarray}}
\newcommand{\eneqn}{\end{eqnarray}}
\newcommand{\benum}{\begin{enumerate}}
\newcommand{\eenum}{\end{enumerate}}

\newcommand{\parens}[1]{\left(#1\right)}
\newcommand{\squared}[1]{\parens{#1}^2}

\newcommand{\tothepow}[2]{\parens{#1}^{#2}}
\newcommand{\prob}[1]{\mathbb{P}\parens{#1}}

\newcommand{\bracks}[1]{\left[#1\right]}
\newcommand{\braces}[1]{\left\{#1\right\}}

\newcommand{\abss}[1]{\left|#1\right|}

\newcommand{\expe}[1]{\mathbb{E}\bracks{#1}}

\newcommand{\indic}[1]{\mathds{1}_{#1}}

\renewcommand{\exp}[1]{\mathrm{exp}\parens{#1}}

\newcommand{\natlog}[1]{\ln\parens{#1}}
\newcommand{\oneover}[1]{\frac{1}{#1}}

\newcommand{\oneoversqrt}[1]{\oneover{\sqrt{#1}}}

\newcommand{\normnot}[2]{\mathcal{N}\parens{#1,\,#2}}

\newcommand{\stdnormnot}{\normnot{0}{1}}

\newcommand{\exponential}[1]{\mathrm{Exp}\parens{#1}}

\newcommand{\zeroonecl}{\bracks{0,1}}

\newcommand{\convp}{~{\buildrel p \over \longrightarrow}~}

\newcommand{\tendn}{\stackrel{n \to \infty}{\longrightarrow}}
\newcommand{\tendt}{\stackrel{t \to \infty}{\longrightarrow}}

\newcommand{\errorrv}{\mathcal{E}}

\newcommand{\Ymax}{Y_{\text{max}}}
\newcommand{\Ymin}{Y_{\text{min}}}

\newtheorem{theorem}{Theorem}
\newtheorem{corollary}{Corollary}[theorem]
\newtheorem{proposition}{Proposition}
\theoremstyle{remark}
\newtheorem{remark}{Remark}[theorem]

\title{Nearly Random Designs with Greatly Improved Balance}

\author[1]{Abba M. Krieger\thanks{Electronic address: \texttt{krieger@wharton.upenn.edu}; Prinicipal Corresponding author}}
\author[2]{David Azriel\thanks{Electronic address: \texttt{davidazr@ie.technion.ac.il}; Corresponding author}}
\author[3]{Adam Kapelner\thanks{Electronic address: \texttt{kapelner@qc.cuny.edu};  Corresponding author}}

\affil[1]{\small Department of Statistics, The Wharton School of the University of Pennsylvania}
\affil[2]{Faculty of Industrial Engineering and Management, The Technion, Haifa, Israel}
\affil[3]{Department of Mathematics,  Queens College, The City  University of New York}

\begin{document}
\maketitle

\begin{abstract}
We present a new experimental design procedure that divides a set of experimental units into two groups so that the two groups are balanced on a prespecified set of covariates and being almost as random as complete randomization. Under complete randomization, the difference in covariate balance as measured by the standardized average between treatment and control will be $O_p(n^{-1/2})$. If the sample size is not too large this may be material. In this article, we present an algorithm which greedily switches assignment pairs. Resultant designs produce balance of the much lower order $O_p(n^{-3})$ for one covariate. However, our algorithm creates assignments which are, strictly speaking, non-random. We introduce two metrics which capture departures from randomization: one in the style of entropy and one in the style of standard error and demonstrate our assignments are nearly as random as complete randomization in terms of both measures. The results are extended to more than one covariate, simulations are provided to illustrate the results and statistical inference under our design is discussed. We provide an open source \texttt{R} package available on \texttt{CRAN} called \texttt{GreedyExperimentalDesign} which generates designs according to our algorithm.
\end{abstract}
\vspace{5cm}
\pagebreak

\section{Introduction}\label{sec:intro}

Consider $2n$ individuals in a balanced two-arm randomized study (e.g. a pill-placebo double-blind clinical trial) that seeks inference for an additive treatment effect. Each subject is placed into the treatment or control group and this procedure is called an \textit{allocation}, an \textit{assignment}, a \textit{design} or a \textit{randomization} and is specified by an \textit{allocation vector}. Upon completion, the sample responses are used to compute an estimate of the average additive treatment effect.

Consider the case of one measured covariate denoted $X_T$ if the subject is in the treatment group and $X_C$ if the subject is in the control group. Denote the $n$ severity levels (measurement values) for the control as  $X_{T,1}, \ldots, X_{T,n}$ and similarly the severity levels for the treatment, $X_{C,1}, \ldots, X_{C,n}$ which are selected via the allocation vector.

The \emph{covariate balance} is commonly defined by $B=|\frac{\overline{X_T} - \overline{X_C}}{s}|$, where $\overline{X_T}$ and $\overline{X_C}$ are the average severities for the treatment and control respectively and $s$ is the standard deviation of the severity over all $2n$ individuals. 

Balance can be affected by the design of the allocation. Under \emph{complete} and balanced randomization, a random treatment-control allocation is chosen from the $\binom{2n}{n}$ possibilities and $B = O_p(n^{-1/2})$. Due to the slow convergence rate, an unlucky draw in a small experiment can result in a large $B$ particularly if $n$ is small. This problem becomes more pathological when there are many more than only one covariate considered. Thus, while complete randomization is consistent in estimating the treatment effect, it may result in substantial imbalance among covariates, and hence, lower power if $X$ features prominently in the response function \citep{Heckman2008}.

Why not find the optimal allocation, which results in $B = O(\frac{\sqrt{n}}{2^{2n}})$, i.e. a far superior balance? This optimal design, a generally unique solution, can be found for very small $n$ but quickly becomes computationally intractable. Each element in the set of designs must be enumerated and its balance checked; this grows exponentially in the sample size and linearly in the number of covariates.

Designing allocations which provide \qu{better} balance on critical covariates than complete randomization (albeit suboptimal) and how one defines \qu{better} is an area of interest dating back to the conception of randomized experiments. Thus, several allocation procedures have been suggested in the literature and are in use e.g. randomized block designs \citep{Fisher1949} and pairwise matching \citep{Greevy2004} yielding $B=O_p(n^{-1})$, rerandomization of the allocation vector $R$ times \citep{Morgan2012} with $B=O_p(n^{-R\sqrt{n}})$. In the abstract, minimizing $B$ can be viewed not as a statistical problem, but as a pure optimization problem. \citet{Bertsimas2015} uses discrete linear optimization to approximate the optimal solution. If there is one covariate, the problem is known as the \qu{number-partitioning problem}. One solution to the problem of note is \citet{Karmarkar1982} who consider a differing algorithm. Their approach was then investigated by \citet{Yakir1996} who demonstrated their algorithm to provide $B = O(n^{-\natlog{n}})$ assuming $X$ has an exponential distribution.

There are also approaches that address balance in the case of sequential experiments (subjects are iteratively randomized) that could, in theory, be implemented here \citep[e.g.][]{Efron1971, Pocock1975, Kapelner2014}. However, we limit our scope in this work to a design applicable only to the non-sequential case (i.e. where all subjects are randomized simultaneously).

Which design should be used in practice? To answer this question, we must first ask, \qu{why is balance important}? By and large, reductions in balance lower the variance of the treatment effect estimator.\footnote{If the covariates are distributed as a multivariate normal, \citet[Theorem 3.2]{Morgan2012} provide an expression for this variance reduction explicitly making no assumptions about the response model. A perfect balance ($B=0$) will reduce the variance by $1-R^2$ where $R^2$ is the proportion of variance explained by a regression of the response on the covariates.},\footnote{A corollary question is, \qu{how should we define balance}? We selected one popular definition of $B$ known to perform well and we explore this topic more in Section~\ref{subsec:choice}.} However, \citet{Kallus2016} builds in an over-arching framework synthesizing experimental design procedures and proves a \qu{no free lunch} style theorem. Given no information about the model (i.e. specification of the functional relationship of the covariates to response), no specific allocation strategy is optimal and hence, one should opt for allocating via complete randomization. He further identifies allocations which provide optimal balance under different modeling assumptions. Surprisingly, many of the non-sequential classic allocation strategies currently in-use (mentioned above) are vindicated as optimal strategies. However, any approach may be arbitrarily suboptimal when the model is misspecified.

In this work, just like the historical procedures, we consider a design strategy that trades the ideal of complete randomization for better balance. We suggest an algorithm that achieves very small balance on the one hand but is also \qu{nearly} as random as complete randomization which should serve as a hedged bet when there is no free lunch. We call our design algorithm \textit{greedy pair-switching} and it proceeds as follows. 

We initialize the allocation vector using complete randomization. We then consider all possible $n^2$ switches of a subject assigned to the treatment group with a subject assigned to the control group. We choose the switch \emph{greedily} to minimize $B$, albeit locally. We repeat these switches until there is no further reduction in $B$. This simple algorithm has some very desirable features. 

First, with one covariate, it can be shown that $B=O_p(n^{-3})$ and with $p$ covariates, $B=O_p(n^{-(1 + 2/p)})$. What is more remarkable is that the algorithm makes very few switches before it stops so that the result is nearly the same as complete randomization. Hence, our designs are robust in the face of nature providing an adversarial model \citep[e.g.][Example 2.1]{Kallus2016} as we show in Section~\ref{sec:consistency}. Further, our algorithm precludes the possibility of the experimenter knowing the assignment of an individual subject, a practicality which must be considered in other allocation procedures \citep[e.g.][]{Efron1971}. Additionally, our procedure can be repeated many times which serves two purposes: (1) to further reduce the balance by taking the minimum of many replicates and (2) to provide many designs which can be used within a permutation test. The latter can test the sharp null hypothesis of no additive treatment effect and can subsequently be inverted to provide a confidence interval of the treatment effect.

The paper proceeds as follows. Section~\ref{sec:alg} analyzes the order of $B$ in our algorithm for the special case of one covariate and then we naturally extend the analysis to $p > 1$ covariates. To measure our procedure's departure from complete randomization, we define two metrics of allocation randomness in Section~\ref{sec:randomness}. We provide simulation results in Section~\ref{sec:sims} to elucidate the theoretical results. Statistical inference of the additive treatment effect in experiments allocated via our design is treated briefly in Section~\ref{sec:inference}. The paper concludes in Section~\ref{sec:conclusion} with a brief discussion and offers future directions. 

\section{Greedy Pair-Switching Algorithm Analysis}\label{sec:alg}

\subsection{The One Covariate Case}\label{sec:one_cov}

We consider the case of one covariate where all measurements are known before assignment. We will create an assignment of $n$ individuals to treatment and the remaining $n$ individuals to control while (a) improving $B$ for one covariate and (b) retaining the desirable property of being nearly random. 

The optimal design that produces the minimum balance may be found by solving an integer programming problem with linear constraints. Hence, the solution $B_\text{min}$ is generally unique. The assignment that maximizes the balance is achieved by giving the $n$ individuals with the largest severity the treatment and the $n$ individuals with lowest severity the control; this is hardly a desirable solution. The balance produced, $B_\text{max}$, will converge to a finite number (if the mean exists) proportional to  $\int_{1/2}^1 F^{-1}(u)du$ where $F$ is the distribution function of severity in the population. 

And so the balance ranges in $\bracks{B_\text{min}, B_\text{max}}$ over all possible divisions of the $2n$ patients into equal groups. There are then an exponential number of solutions, $\binom{2n}{n}$, which is of order $\frac{2^{2n}}{\sqrt{n}}$ in this range. This argument suggests that the computationally intractable optimal solution will in general have balance that is of order $\frac{\sqrt{n}}{2^{2n}}$ if the distribution of severity is absolutely continuous.

We use this reasoning to derive the order of the balance under our proposed \textit{greedy pair-switching}. Since there are an exponential number of solutions all of which have finite balance (and only one solution that is minimal), there must be many solutions that are sufficiently small for practical purposes, say of order $n^{-k}$ for some $k$. The greedy algorithm presented in the introduction finds such a division.  We prove below that with one covariate the resulting balance using the greedy algorithm is of order $\frac{1}{n^3}$ and there are many divisions with balance of this order. The greedy algorithm converges to dissimilar solutions depending on the initial randomized allocation vector. Hence our algorithm is also \textit{nearly} random, as will be formalized in Sections~\ref{sec:randomness} and \ref{sec:sims}. 

The reason the algorithm is nearly random is because it almost surely terminates after very few switches.  The intuition is that after relatively few switches the  $\sum_{i=1}^n X_{T,i} - \sum_{i=1}^n X_{C,i}$ goes from order $\sqrt{n}$ to a constant. Once  the difference of the sums is a constant, $c$ , then all we need is a pair $(X_i,Y_j)$ to be within $c+\frac{d}{n^2}$ as the difference of the sums being of order $\frac{1}{n^2}$ implies that the averages differ by order $\frac{1}{n^3}$. The argument below shows that such a pair exists with high probability, particularly if $d$ is sufficiently large.

More formally, let $X_T$ and $X_c$ be independent and identically distributed absolutely continuous random variables with uniformally continuous density function $f$ and distribution function $F$. Let

\bneqn\label{eq:Ace}
A(c,\epsilon)=\braces{(x_T, x_C) : c - \epsilon \leq \abss{x_T - x_C} \leq c + \epsilon}
\eneqn

\noindent and 

\bneqn\label{eq:Pce}
P(c,\epsilon) = \prob{(X_T, X_C) \in A(c,\epsilon)}.
\eneqn

\noindent We can then define $P(c)$ as $\epsilon$ vanishes, 

\begin{eqnarray}\label{eq:p_c}
P(c) &=& \lim _{\epsilon \rightarrow  0} \frac{P(c,\epsilon)}{2\epsilon} \nonumber \\  
&=&  \lim _{\epsilon \rightarrow  0} \frac{\int_x F(x+c+\epsilon)-F(x+c-\epsilon)}{2\epsilon}f(x)dx \nonumber \\ 
&=& \int_\reals f(x+c)f(x)dx.
\end{eqnarray}

It is an elementary exercise to verify that if $X \sim U(0,1)$ then $P(c)=1-c$, if $X \sim \exponential{1}$ then $P(c)=\frac{1}{2}e^{-c}$ and if $X \sim \stdnormnot$ then $P(c)=\half\oneoversqrt{\pi}e^{-c^2/4}$.

Now consider $X_{T,1}, \ldots, X_{T,n}, X_{C,1}, \ldots, X_{C,n} \iid f$. Let $B_{i,j}$ be the event that $(X_{T,i},X_{C,j}) \in A(c,\epsilon)$ Note that there are $n^2$ events $B_{i,j}$. If these events were independent and sufficiently small $\epsilon$, say $\epsilon=\frac{d}{n^2}$, then the number of such events that will occur is asymptotically distributed as a Poisson($\lambda = 2dP(c)$).

But the issue is that since there are $n^2$ comparisons, but only $2n$ random variables, these $B_{i,j}$ are dependent. However, we do not need to know the number of these events that will occur, but rather only whether there is any such event. If there is such an event, then we can switch these two observations and move the sum of the differences in the severity from a constant $c$ to something of order $n^{-2}$. We present this formally below. \\

\begin{theorem}
\label{thm:main}
For any value of $c$, and any probability $\gamma$, there exists an $N$ such that for any $n>N$, the probability that there exists a pair $(X_{T,i}, X_{C,j})$ that satisfies $\prob{(X_{T,i},X_{C,j}) \in A(c,\epsilon)}>\gamma$ where $\epsilon= \frac{1}{2n^2P(c)(1-\gamma)}$.
\end{theorem}

\begin{proof}
\noindent To prove this, we make use of the following result by \citet{decaen1997}, 
\begin{equation}\label{bound}
\prob{\bigcup_{i,j}B_{i,j}} \ge \sum_{i,j}\frac{\prob{B_{i,j}}^2}{\sum_{k,l}\prob{B_{i,j} \cap B_{k,l}}}.
\end{equation}
Note that each of the $n^2$ terms in the sum on the right hand side of (\ref{bound}), has the same value, which we now evaluate. To this end, consider 
four cases for $\prob{B_{i,j} \cap B_{k,l}}$: 
\begin{itemize}
\item If $i \ne k$ and $j \ne l$ then the events are independent. Hence

\beqn
\lim_{\epsilon \rightarrow 0}\frac{\prob{B_{i,j} \cap B_{k,l}}}{4\epsilon^2} = P^2(c).
\eeqn

\item If $i=k$ and $j=l$ then $\lim_{\epsilon \rightarrow 0}\frac{\prob{B_{i,j}}}{2\epsilon}=P(c)$. 
\item If $i=k$ and $j \ne l$ then $\prob{B_{i,j} \cap B_{i,l}}=\prob{B_{i,j}}\prob{B_{i,l}|B_{i,j}}$.

\noindent Consider $\prob{B_{i,l}|B_{i,j}}$. What effects $B_{i,l}$ is the density function of $X_i$ given $B_{i,j}$. But 

\beqn
\prob{X_i \le x~|~B_{i,j}}= \frac{\displaystyle \int_{-\infty}^x (F(z+c+\epsilon)-F(z+c-\epsilon))f(z)dz}{\displaystyle\int_\reals (F(z+c+\epsilon)-F(z+c-\epsilon))f(z)dz}.
\eeqn

\noindent Hence,

\begin{equation}\label{density}
\lim_{\epsilon \rightarrow 0} f_{X_i|B_{i,j}}(x)= \frac{f(x+c)f(x)}{\displaystyle \int_\reals f(z+c)f(z)dz}.
\end{equation}

\noindent In order to evaluate $P^+(c) \equiv \lim_{\epsilon \rightarrow 0}\frac{\prob{B_{i,l}~|~B_{i,j}}}{2\epsilon}$, we first calculate

\beqn
\prob{(X_{T,i},X_{C,l}) \in A(c,\epsilon)~|~X_{T,i} = x},
\eeqn 

\noindent then weight these by the density of $X_i$ given in (\ref{density}) and divide by $2\epsilon$. This yields 

\begin{equation}\label{condp}
P^+(c)=\frac{\int_\reals f^2(x+c) f(x) dx}{P(c)}.
\end{equation}

\item If $i \ne k$ and $j=l$ we can follow a similar argument to the previous case to yield

\begin{equation}\label{condn}
P^-(c) \equiv \lim_{\epsilon \rightarrow 0}\frac{\prob{B_{i,j}|B_{k,j}}}{2\epsilon}= \frac{\int_\reals f^2(x-c)f(x)dx}{P(c)}.
\end{equation}
\end{itemize}

From (\ref{bound}) after dividing the numerator and denominator by $4n^2 \epsilon^2$ this yields

\beqn
\lim_{n \rightarrow \infty} \prob{\bigcup_{i,j}B_{i,j}} \ge \lim_{n \rightarrow \infty}\frac{P^2(c)}{\frac{(n-1)^2}{n^2}P^2(c)+\frac{n-1}{n^2}P(c)(P^+(c)+P^-(c))+\frac{1}{2n^2\epsilon}P(c)}.
\eeqn

Letting $\epsilon=\frac{d}{n^2}$, the limit of the denominator in the above expression is $P^2(c)+\frac{P(c)}{2d}$. If we set $\gamma=\frac{P^2(c)}{P^2(c)+\frac{P(c)}{2d}}$, then $d=\frac{\gamma}{2(1-\gamma)P(c)}$. The result follows by making $d$ a bit larger by ignoring $\gamma$ in the numerator.
\end{proof}

~

\begin{remark} 
It is easy to see that $P^+(c) \ge P(c)$ as $\int f^2(x+c)f(x)dx > (\int f(x+c)f(x)dx)^2$. Similarly, $P^-(c) \ge P(c)$. This is intuitive because if $(X_{T,i},X_{C,k}) \in A(c,\epsilon)$ this implies that the two random variables are more likely to be in the fatter part of the distribution, and as a result, it is more likely that another $X_C$ is close to $X_{T,i}$ and another $X_T$ is close to $X_{C,k}$. For the three special cases considered above, $P^+(c)= P^-(c)=1$ if the distributions are uniform; $P^+(c)=\frac{2}{3}e^{-c}$ and $P^-(c)=\frac{2}{3}$ if the distributions are exponential with mean of $1$; $P^+(c)=P^-(c)=\frac{1}{\sqrt{3\pi}}e^{-\frac{1}{12}c^2}$ if the distributions are standard normal.
\end{remark}

~ 

\begin{remark}\label{rem:thm1b}
Theorem~\ref{thm:main} still holds true where instead of considering all $n^2$ events $B_{i,j}$, we consider only pairs of a group of size $n^2-n$. This is the case since now the right hand side of (\ref{bound}), will include $n^2-n$ summands but these asymptotically behave like $n^2$ summands. 
\end{remark}

Theorem~\ref{thm:main} suggessts two corrolaries:\\

\begin{corollary}
The number of switches that will be required is of order no greater than $\sqrt{n}$.
\end{corollary}

\begin{proof}
This is immediate as after at most an amount proportional to $\sqrt{n}$ switches, the difference in the sums will be reduced to an amount of  $O(1)$. One more switch brings it to an amount of $O(n^{-2})$ by the above theorem and Remark~\ref{rem:thm1b}. Finally there are only a finite number of pairs of $(X_{T,i},X_{C,j})$ that satisfy $|X_{T,i} - X_{C,j}| \le \frac{k}{n^2}$ for any $k$. 
\end{proof}

~

\begin{corollary}
If the covariate has finite variance the greedy algorithm produces a balance, $B$  that is $O_p(n^{-3})$.
\end{corollary}

\begin{proof}
Since the difference in the sum is asymptotically normally distributed with mean of $0$ and standard deviation of $\sqrt{2n}$ the first steps of the algorithm will be to switch the largest value of the subset with smaller sum for the smallest value of the subset with larger sum. Eventually the balance will be some constant $c$. Theorem~\ref{thm:main} and Remark~\ref{rem:thm1b} now imply that there will be one possible switch that can reduce the difference of the sum from a constant to a value of order $n^{-2}$, hence the balance, $B$ will be of order $n^{-3}$.
\end{proof}

\subsection{The Multiple Covariate Case}\label{sec:mult_cov}

This section mirrors the previous section except that there are now $p$ covariates rather than one covariate. The primary result in the previous section was that the greedy pair-switching heuristic produces a solution resulting in $B = O_p(n^{-3})$ when there is one covariate. We now consider random assignments of $2n$ individuals to treatment and control to balance $p$ covariates. Assume that the random assignment results in covariates for the treatment that  are denoted by  $X_{T,i,j}$ and similarly for the control of $X_{C,i,j}$ where $i=1,\ldots,n$ indexing the subjects and $j=1, \ldots,p$ indexing the covariate measurements. We assume that these random vectors $\bv{X}_{T,i}$ and $\bv{X}_{C,i}$ are independent across the $2n$ individuals. Once we observe the covariates for the $2n$ individuals we can standardize each covariate without loss of generality so that $\sum_{i=1}^n x_{T,i,j} + x_{C,i,j} = 0$ and $\sum _{i=1}^n x_{T,i,j}^2 +x_{C,i,j}^2=2n-1$ for each $j=1,\ldots,p$. The standardized difference-in-means balance metric can then be alternatively expressed as

\[
B= \frac{2}{n}\sum_{j=1}^p \abss{\sum_{i=1}^n x_{T,i,j}}.
\]

The result in the previous section hinged on the fact that with high probability there will be an individual in the  treatment group and one in the  control group  that differ on the covariate by an amount $c \pm \epsilon$ where $\epsilon = O(n^{-2})$. Now when we switch a treated subject for a control subject, all $p$ covariates are of issue, so being close on one covariate will not necessarily guarantee being close on other covariates. The primary observation is that we can find a pair of individuals who differ by an amount $c_j \pm \epsilon$ for all $j=1,\ldots,p$, but now $\epsilon = O_p(n^{-2 / p})$ so that $\epsilon^p = O_p(n^{-2})$.

Generalizing definition~\ref{eq:Ace} to $p$ covariates, we let

\beqn
A(\c,\epsilon)=\braces{(\bv{x}_T, \bv{x}_C) : \forall j~c_j - \epsilon \leq \abss{x_{T,j} - x_{C,j}} \leq c_j + \epsilon}
\eeqn

and $P(\c,\epsilon)$ is as in definition~\ref{eq:Pce} now with vectors $\bv{X}_{T}$, $\bv{X}_{C}$ and $\c = \bracks{c_1, \ldots, c_p}$. It follows that 

\beqn
\lim_{\epsilon \rightarrow 0}\frac{P(\c,\epsilon)}{(2\epsilon)^p} =
\lim_{\epsilon \rightarrow 0}\int\displaylimits_{\x \in \reals^p} \frac{F(\x+\c+\epsilon) - F(\x+\c-\epsilon)}{(2\epsilon)^p}f(\x)d\x =
\int\displaylimits_{\x \in \reals^p}f(\x+\c)f(\x)d\x.
\eeqn

The key is that the integral on the right-hand side is a constant, $\prob{\c}>0$.

~

\begin{theorem}
Let $B_{i,j}$ be the event that $(\bv{X}_{T,i}, \bv{X}_{C,j}) \in A(\c,\epsilon)$. For any $0 < \gamma < 1$, there exists an $N$ such that $\forall n>N$ the probability that at least one of the $n^2$ $B_{i,j}$ occurs exceeds $\gamma$,
where $\epsilon=\half \tothepow{\frac{1}{n^2 \prob{\c} (1-\gamma)}}{1/p}$.
\end{theorem}

\begin{proof}

We parallel the first step in Theorem~\ref{thm:main} to obtain

\footnotesize
\[
\prob{\bigcup_{i,j}B_{i,j}} \ge \frac{n^2\prob{B_{1,1}}^2}{(n-1)^2\prob{B_{1,1}}^2 +(n-1)\prob{B_{1,1}}\prob{B_{1,2}\mid B_{1,1}}+(n-1)\prob{B_{1,1}}\prob{B_{2,1}\mid B_{1,1}}+\prob{B_{1,1}}}.
\]
\normalsize

Dividing numerator and denominator by $n^2(2\epsilon)^p$ and letting $n \rightarrow \infty$ implies that the right-hand side is

\beqn
\frac{\prob{\c}^2}{\prob{\c}^2 + \displaystyle\lim _{ n \rightarrow \infty} \frac{n-1}{n^2} \prob{\c} \frac{\prob{B_{1,2}\mid B_{1,1}}}{(2\epsilon)^p}+\displaystyle\lim _{ n \rightarrow \infty} \frac{n-1}{n^2}\prob{\c}\frac{\prob{B_{2,1}\mid B_{1,1}}}{(2\epsilon)^p} +\displaystyle\lim_{n \rightarrow \infty} \frac{\prob{\c}}{n^2(2\epsilon)^p}}.
\eeqn

But

\beqn
\lim _{\epsilon \rightarrow 0}f(\x~|~B_{1,1})=\frac{f^2(\x+c)f(\x)}{\int_{\reals^p} f^2(\z+c)f(\z)dz}
\eeqn

This implies that $\lim _{\epsilon \rightarrow 0}\frac{\prob{B_{1,2}\mid B_{1,1}}}{(2\epsilon)^p}$ and $\lim _{\epsilon \rightarrow 0}\frac{\prob{B_{2,1}\mid B_{1,1}}}{(2\epsilon)^p}$  both go to a constant so the two middle terms in the denominator goes to $0$. The result follows by straightforward algebra.
\end{proof}

This implies that the final rate for general number of covariates $p$ is $O_p(n^{-(1 + 2/p)})$.

According to \citet{Bertsimas2015} the optimal rate is exponentially small. But the optimal rate produces a unique solution and hence is not random. Our rate is small in an absolute sense, but relatively much larger than the optimal rate. This suggests that there are many solutions of this rate. Choosing one of these many solutions still results in a division into treatment and control that is small while maintaining some of the randomness from randomization. Definining \qu{randomness} is treated in the Section~\ref{sec:randomness} and Section~\ref{sec:inference} posits why it is important. Before we discuss these topics, we note that we have developed enough theory to understand how our algorithm would perform if $B$ were defined differently.

\subsection{Choice of balance function}\label{subsec:choice}

There are many ways to measure balance. \citet{Franklin2014} details ten popular choices and compares each in a simulation study of bias. Our choice of objective $B$, the sum of absolute standardized average differences is found to perform among the best.

However, would the theoretical performance of our greedy pair-switching algorithm be invalidated if we chose another definition of $B$? \citet{Morgan2012}, \citet{Greevy2004} and others define $B$ as Mahalanobis distance for instance. We show here that for the choice of Mahalanobis distance, we would have the same results.

Consider a regression where the variable of interest is the treatment effect (as compared to a control) which is expressed in effects coding as a $1$ for treatment and $-1$ for control  (without loss of generality). It is straightforward to show that the estimated standard error for the treatment effect is 

\beqn 
s_{b_T} = s_e \frac{1}{\sqrt{n}}\sqrt{\frac{2}{1-Q_p\frac{2}{2n-1}}}
\eeqn

\noindent where $Q$ is the Mahalanobis distance between the $p$-dimensional $\overline{x_T}$ and $\overline{x_C}$ and as such has mean of $p$ and $s_e$ is the classical estimate of the standard error of the response noise. But making each of the $p$ averages between treatment and control to be of order $O_p(n^{-2/p})$ ensures that if $B$ was defined as the Mahalanobis distance, the results will be of the same order. 

Understanding the theoretical performance of our greedy pair-switching algorithm for other specific measures of balance is left to future work. However,  consider when $B$ is ``nice'', then the first order approximation is linear and therefore we expect that balancing the $L_1$ distance as we did here, will result in a similar order of balance discussed herein.

\section{Two Measures of Randomness}\label{sec:randomness}

The greedy heuristic produces balance between the treatment and control that is very small, $B = O_p(n^{-3})$ in the case of one covariate. In fact, there are solutions that guarantee much smaller balance, but these solutions are more deterministic. We consider the behavior of various design algorithms with regard to how random they are. In our context, \textit{complete randomness} is defined to be all $\binom{2n}{n}$ possible assignments to treatment and control are equally likely. 

\subsection{An Entropy Metric of Randomness}

A logical metric to capture deviations from complete randomness is entropy over the assignments computed via the probabilities of all $\binom{2n}{n}$ possible assignments. An assignment that is not completely random will have non-equal probabilities over the possible divisions of the $2n$ individuals into treatment and control. 

The actual entropy will be difficult to determine. We take the approach of examining the bivariate distributions of each pair of subjects within an assignment. Specifically, let $S=\{(i,j)~|~1\,\le\,i\,<\,j\,\le\,2n\}$ be all pairs of individuals and $|S|=\binom{2n}{2}$, the number of points in $S$. In a random solution each pair $s \in S$  has the same probability of $s_n=\frac{n-1}{2n-1}$ of being in the same group. An algorithm can be characterized by $p_s$, which is the probability that the pair of individuals in $s$ are in the same group for all $s \in S$. We can then define pairwise entropy as 

\begin{equation}
\label{eq:entropy}
E_n=\frac{1}{|S|}\frac{\sum _{s \in S} p_s \natlog{p_s} +(1-p_s) \natlog{1-p_s}}{s_n \natlog{s_n} + (1-s_n)\natlog{1-s_n}}
\end{equation}

where, by common convention, $0 \times \natlog{0} = 0$. Note that $0 \le E_n \le 1$, where $E_n=0$ if the assignment is deterministic so that $p_s$ is either $0$ or $1$ for all $s$ and $E_n=1$ if $p_s=s_n$ for all $s$ (i.e. under complete randomness).

\subsection{A Standard Error Metric of Randomness}

Under complete randomness the pair probabilities are all the same and when the algorithm is deterministic the pair probabilities are either $0$ or $1$. Thus another logical metric is the standard deviation of these $p_s$'s from the probability under complete randomness, $p_s = \frac{n-1}{2n-1}$ as below: 

\beqn
s(p_s) = \sqrt{\oneover{\binom{n}{2}} \sum_{s \in S} \squared{p_s-\frac{n-1}{2n-1}}}.
\eeqn

The standard deviation is $0$ if completely random; the larger the standard deviation, the further the assignment is from complete randomness. Note that the maximum standard deviation occurs when the algorithm is deterministic in which case  $n(n-1)$ of the probabilities are $1$ and $n^2$ of the probabilities are $0$. Thus, we can scale the above expression to be between $\zeroonecl$ by dividing by this maximum. Straightforward algebra then defines the metric:

\begin{equation}\label{eq:se}
D_n=\frac{1}{n}\sqrt{\frac{2n-1}{n-1}\sum_{s \in S} \squared{p_s-\frac{n-1}{2n-1}}},
\end{equation}

It is likely that the two measures will preserve order; i.e., if $p_s$ has higher entropy than $q_s$  then the  standard deviation of the $p_s$ is smaller than the standard deviation of the $q_s$. However, this result is not true in general. Let $p=(.3,.3,.9)$ and $q=(.153,.5,.847)$. The entropy, $-\frac{\sum_{i=}^3 p_i\log(p_i)+(1-p_i)\log(1-p_i)}{3\log(2)}=.744$ and twice (to scale from $0$ to $1$)  the standard deviation of $p$ is $0.693$. The corresponding entropy and standard deviation of   $q$ are $0.745$ and $0.694$. So $q$ is more random based on entropy and $p$ is more random based on the variability of the probabilities.

\subsection{The Randomness of our Design Algorithm}

We now want to show that the greedy heuristic is nearly random. To this end, assume that there is an algorithm $\mathcal{A}$ that begins with randomly choosing $n$ observations from the available $2n$ items for treatment assignment. Furthermore, there exists indices $A \subseteq \{1,\ldots, n \}$ where all subjects indexed by $A$ retain the assignment based on randomization, where $A$ is fixed; that is, it does not depend on the randomization. The following theorem follows straightforwardly. 

\begin{theorem}\label{thm:random}
The algorithm, $\mathcal{A}$ satisfies $\displaystyle \lim_{n \rightarrow \infty} E_n = 1$ and $\displaystyle \lim_{n \rightarrow \infty} D_n=0$ if $\displaystyle \lim_{n \rightarrow \infty} \frac{|A|}{2n}=1$.
\end{theorem}

\begin{proof}
Let $S_A$ be the subset of $S$ where $i \in A$ and $j  \in A$.  If $s \in S_A$ then $p_s=s_n$. But the number of points in $S_A$ is $\binom{|A|}{2}$. Hence 
\[
\lim_{n \rightarrow \infty} E_n \ge \lim_{n \rightarrow \infty} \frac{\binom{|A|}{2}}{|S|}= \lim_{n \rightarrow \infty}\frac{|A|(|A|-1)}{2n(2n-1)} = 1.
\]
Similarly the terms in Definition~\ref{eq:se} are zero for all $s \in S_A$ and are maximally $\big(\frac{n}{2n-1}\big)^2 \le 1$ otherwise. Therefore,
\[
D_n \le \frac{1}{n}\sqrt{\frac{2n-1}{n-1} |\overline{S_A}|}.
\]
But $\lim _{ n \rightarrow \infty} \frac{ |\overline{S_A}|}{n^2} = 0$, therefore, $\lim _{n \rightarrow \infty}D_n=0$.
\end{proof}

Although it is difficult to show that the greedy algorithm satisfies the condition of the theorem, it undoubtedly does. Consider a modification of the algorithm. We can condition on the $2n$ values and consider the set of points $E$ that are the largest and smallest $b\sqrt{n}$ observations for some suitably chosen value of $b$. We begin by randomly selecting $n$ values to the treatment. We run a modified greedy pair-switching algorithm only considering switching items that are in $E$. Once the sum of the difference in the values of the two groups is less than $c$, make one additional switch. \\

\begin{corollary}
For any value of $\gamma$ there exists and $N$  such that for all $n>N$, the probability that the balance is $O(n^{-3})>\gamma$ if the modified greedy heuristic is employed when there is one covariate.
\end{corollary}

\begin{proof}
This follows from Theorem~\ref{thm:random} and the construction of the modified greedy heuristic.
\end{proof}

Finally, one can ask how random are other commonly used algorithms. 

Consider in the case of one covariate the simple algorithm that creates pairs of observations from among the $2n$ observations where the largest two values form the first pair, the third and fourth largest values the second pair and so forth. If one randomly assignms an individual in each pair to the treatment group and the other to the control group, $B=O_p(n^{-1})$. This assignment is not completely random either. For example, the value of $D_n=\frac{1}{\sqrt{n}}\sqrt{\frac{n+1}{2(n-1)}}$ which goes to zero at rate $n^{-\frac{1}{2}}$,  but is not zero. This algorithm cannot directly be extended to $p$ covariates, but there is software (e.g. \texttt{optmatch} by \citealt{Hansen2006}) that may be used. It should be noted that the purpose of matching is not only concerned with balance. It also is justified as a way to guard against misspecification and potentially unobserved covariates. 

If one uses the rerandomziation scheme of \citet{Morgan2012}, the balance is of $B=O(n^{-R\sqrt{n}})$ where $R$ is the number of randomizations. Thus, this approach requires many rerandomzations before its rate is of the same order as the greedy heuristic. Furthermore, this approach is also not perfectly random as defined by the measures presented here.

\section{Simulation Results}\label{sec:sims}

To illustrate our main results we turn to simulations. All simulations herein were performed with \texttt{GreedyExperimentalDesign}, an \texttt{R} package available on \texttt{CRAN} whose core is implemented in \texttt{Java} for speed. We first illustrate the main claim of Section~\ref{sec:one_cov} that the greedy pair switching algorithm creates designs with vastly improved balance as measured by $B$, the absolute difference in the average values of the one covariate among treatment and control. 

Table~\ref{tab:1} shows the first twenty iterations of $n=50$ (with $p=1$) by initial and final balance as well as the number of switches. Note that among 20 repetitions balance is diminished by 5-7 orders of magnitude. Since the covariate is simulated as standard normal, the resulting final balances are all less than 5 $\times 10^{-5}$ standard deviations. This is a miniscule amount; and it is likely smaller than the precision of the covariate measurement.

\begin{table}[h]
\centering
\begin{tabular}{ccc}
  \hline
Initial Balance& \# of switches & Final Balance \\ 
  \hline
0.37 & 3 & 0.0000004 \\ 
  0.14 & 3 & 0.0000007 \\ 
  0.25 & 5 & 0.0000015 \\ 
  0.22 & 4 & 0.0000021 \\ 
  0.24 & 4 & 0.0000036 \\ 
  0.01 & 3 & 0.0000037 \\ 
  0.20 & 2 & 0.0000042 \\ 
  0.21 & 4 & 0.0000056 \\ 
  0.23 & 2 & 0.0000064 \\ 
  0.06 & 1 & 0.0000072 \\ 
  0.08 & 2 & 0.0000075 \\ 
  0.16 & 2 & 0.0000077 \\ 
  0.11 & 2 & 0.0000103 \\ 
  0.05 & 2 & 0.0000109 \\ 
  0.08 & 1 & 0.0000126 \\ 
  0.22 & 3 & 0.0000138 \\ 
  0.05 & 2 & 0.0000174 \\ 
  0.51 & 4 & 0.0000205 \\ 
  0.09 & 2 & 0.0000282 \\ 
  0.13 & 1 & 0.0000326 \\ 
   \hline
\end{tabular}
\caption{Twenty iterations of greedy pair switching on a dataset of $n=50$ of standard normal covariate: starting balance, number of switching and final balance ordered by best final balance.}
\label{tab:1}
\end{table}

How does the greedy pair switching algorithm compare to the optimal balance? Since finding the optimal balance is an exponentially difficult problem in $n$, we consider the case when  $n=14$ so that the optimal can be found by complete enumeration. We generate $28$ observations from a standard normal distribution and  display the balance results for each of five random assignments followed by the balance results of the greedy pair-switching algorithm and finally, the optimal balance in Table~\ref{tab:2}. We note that while optimal performs at $\frac{\sqrt{n}}{2^{2n}}$ (i.e. $\approx 10^{-8}$ here) and greedy pair switching performs at approximately the $\frac{1}{n^3}$ (i.e. $\approx 10^{-4}$ here), the difference is likely once again of little practical difference when considering the standard deviation of the covariate and likely the measuring instrument used to assess the covariate. The difference between optimal and our procedure is further likely to be dwarfed when considering inexplicable noise in the response function.

\begin{table}[ht]
\centering
\begin{tabular}{rrr}
  \hline
Initial Balance& \# of switches & Final Balance \\ 
  \hline
0.06 & 2 & 0.000130390 \\ 
  0.15 & 1 & 0.000457781 \\ 
  0.03 & 2 & 0.000772729 \\ 
  0.92 & 3 & 0.000909507 \\ 
  0.82 & 3 & 0.001875810 \\ \hline
\multicolumn{2}{r}{Optimal}  & 0.000000015 \\ 
   \hline
\end{tabular}
\caption{Five iterations of greedy pair switching on a dataset of $n=14$ of a standard normal covariate: starting balance, number of switching and final balance ordered by best final balance. Bottom row: exhaustive search of the $\binom{28}{14} = 40,116,600$ possible allocation vectors to find the optimal balance. Our software package runs this search asynchronously and in parallel.}
\label{tab:2}
\end{table}

We also illustrate that the covariate distribution does not matter when considering the exponential reduction in balance in Figure~\ref{fig:diff_dist}. Here, we simulate realizations of normal, exponential and uniform for values of $\log_{10}(n)$ ranging from 1-2.5. See discussion just after (\ref{eq:p_c}) about how changing this distribution is not likely to matter. 

\begin{figure}[htp]
\centering
\includegraphics[width=3.8in]{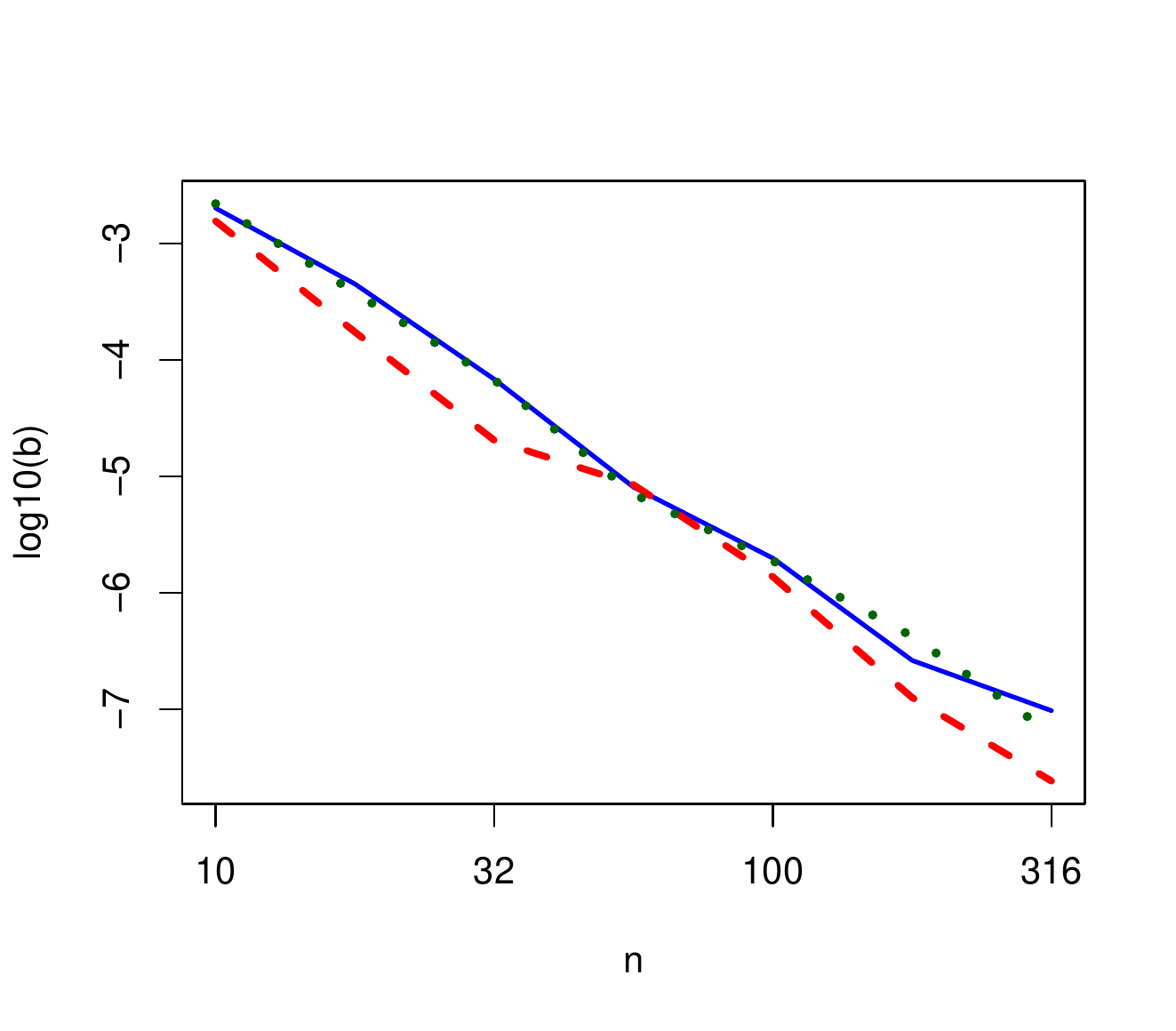}
\caption{$\log_{10}(B)$ by $n \in 10^{\braces{1, ~1.25,~ \ldots,~ 2.5}}$ for the standard normal design (solid blue), standard exponential design (dashed red) and standard uniform design (dotted green) averaged over $r=1000$ replicates.}
\label{fig:diff_dist}
\end{figure}%

To explore the rate relationship proven in Section~\ref{sec:mult_cov}, we simulate $n \in 10^{\braces{1, ~1.25,~ \ldots,~ 2.5}}$ and $p \in \braces{1,2,5,10,40}$. The log-log plot of balance is displayed in Figure~\ref{fig:np_a}. Note the linear relationship which becomes flatter as the number of covariates increases. We also explore the number of greedy pair switches as a function of $n$ and $p$ in Figure~\ref{fig:np_b} but we do not further investigate its relationship.

\begin{figure}[htp]
\begin{subfigure}[b]{0.5\linewidth}
\centering
\includegraphics[width=3.2in]{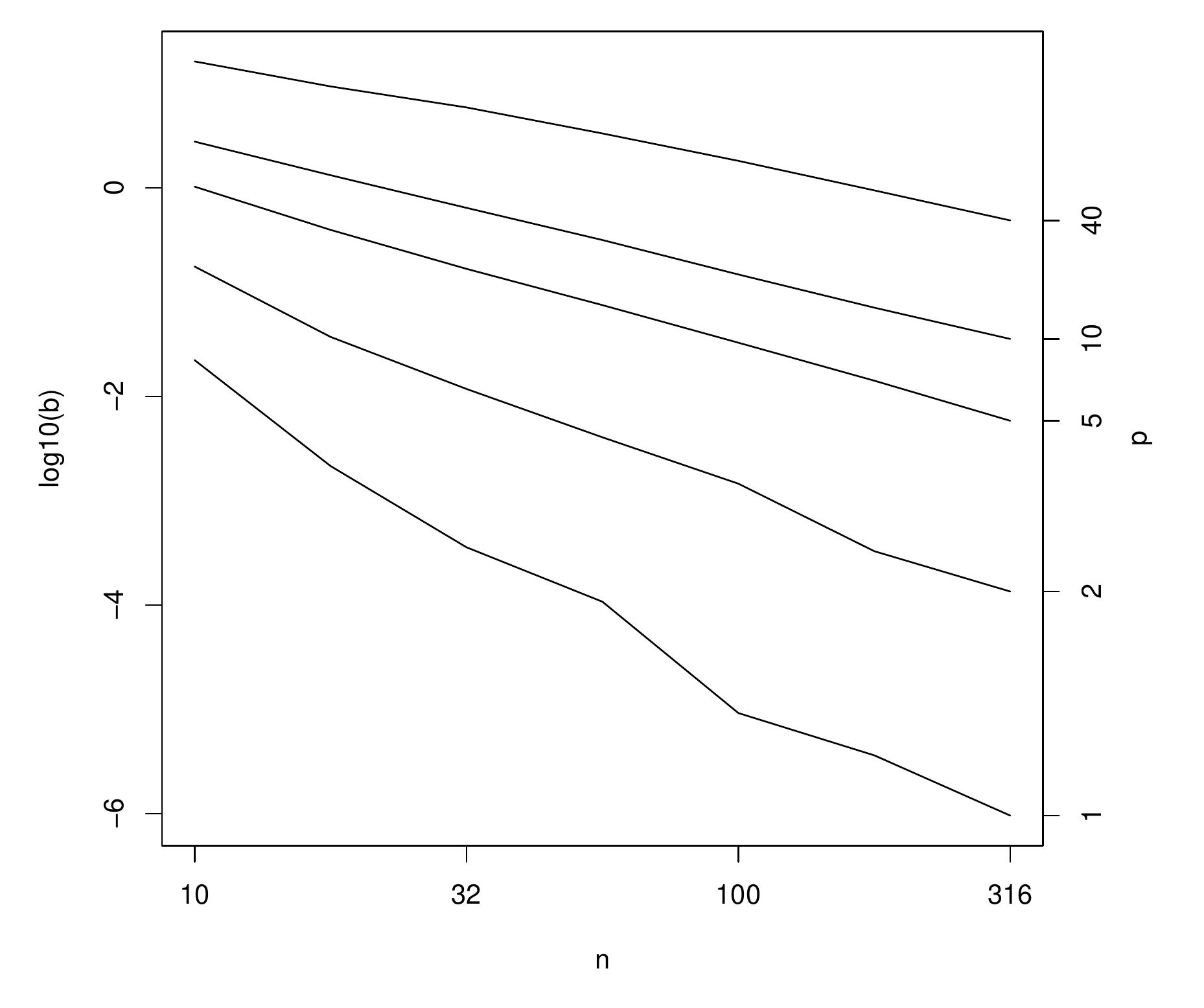}
\caption{}
\label{fig:np_a}
\end{subfigure}%
\begin{subfigure}[b]{0.5\linewidth}
\centering
\includegraphics[width=3.2in]{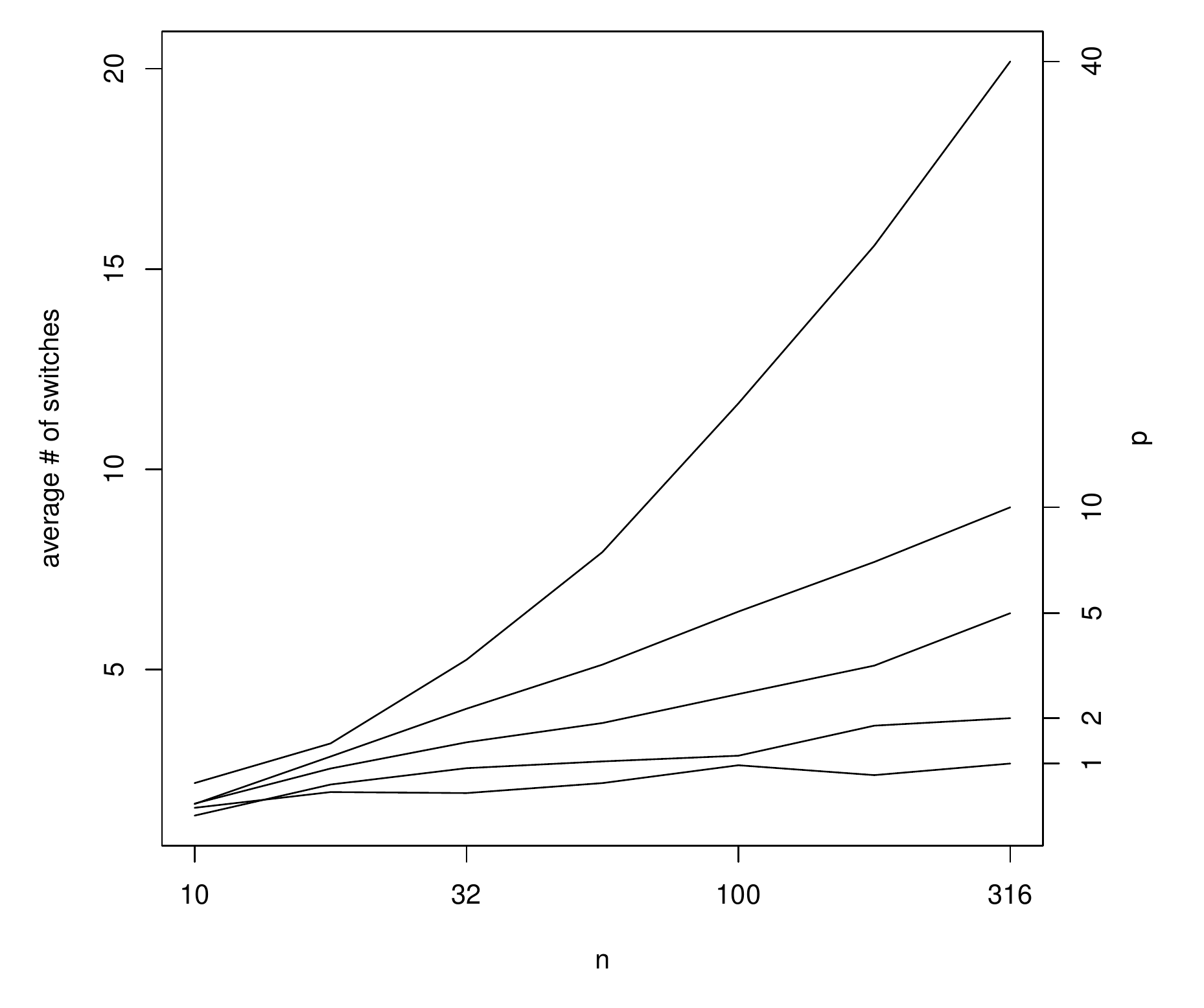}
\caption{}
\label{fig:np_b}
\end{subfigure}
\caption{Left: $B$ by $n$ and $p$. Right: number of switches in the by $n$ and $p$. The average of $r=1000$ repetitions for both plots.}
\label{fig:np}
\end{figure}

Recall the general rate of our algorithm with $p$ covariates is $O_p(n^{-(1 + 2/p)})$ which implies the following linear relationship with slope coefficients: 

\beqn
\natlog{B} = c(n,p,X) + (\textbf{-1}) \natlog{n} + \textbf{(-2)}\frac{1}{p} \natlog{n} + \errorrv
\eeqn

\noindent where $c(n,p,X)$ is a constant and $\errorrv$ is noise (distribution unknown) due to the random starting allocation vectors and the distribution of the covariate(s). We provide regression confirmation of this rate in Table~\ref{tab:regression} which confirms the -1 coefficient on the log sample size and the -2 coefficient on log sample size crossed with inverse covariate dimension. In practice, remember our claims are asymptotic and we have noticed in many unshown simulations slower convergence than displayed here. We leave exploration of the differential rates of convergence to future work.

\begin{table}[h] 
\centering 
\begin{tabular}{lc} 
\\$R^{2}$ = 0.966 & log OLS coef. $\pm$ $2 \times $s.e. \\ 
\hline \\[-1.8ex] 
 $c(X)$ & ~3.700 $\pm$ 0.176 \\ 
 $p^{-1}$ & -3.681 $\pm$ 0.244   \\ 
 $\natlog{n}$ & \textbf{-1.042} $\pm$ 0.042 \\ 
 $\natlog{n} \times p^{-1}$ & \textbf{-2.063} $\pm$ 0.082  \\ 
\end{tabular} 
\caption{Regression of log balance on log $n$ cross inverse $p$ for $r=100$ and for the $n \times p$ simulation design of Figure~\ref{fig:np_a} for a total of $3,500$ regressed observations.} 
\label{tab:regression} 
\end{table} 

Finally, we demonstrate our titular claim --- our algorithm provides impressive balance performance while not sacrificing randomness in the experimental design. Figure~\ref{fig:entropies} plots randomness as measured by the entropy metric of Definition~\ref{eq:entropy} and Figure~\ref{fig:ses} plots randomness as measured by the standard error metric of Definition~\ref{eq:se} both as a function of the $n$ and $p$ found in Figure~\ref{fig:np}. Note that by $n \approx 100$, our designs are nearly indistinguishable from complete randomization as gauged by both metrics of randomization.

\begin{figure}[htp]
\begin{subfigure}{0.5\linewidth}
\centering
\includegraphics[width=3.2in]{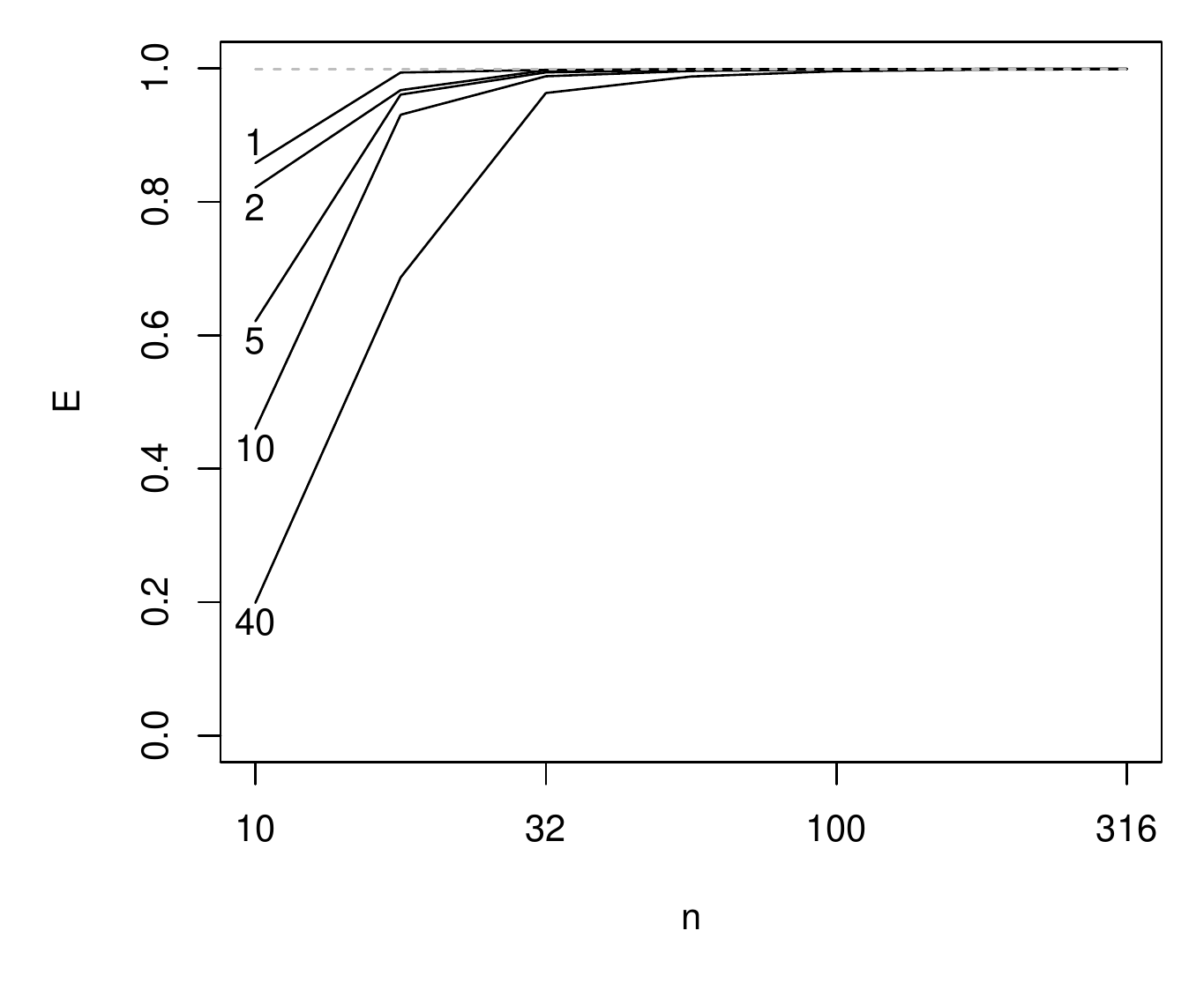}
\caption{$E$}
\label{fig:entropies}
\end{subfigure}%
\begin{subfigure}{0.5\linewidth}
\centering
\includegraphics[width=3.2in]{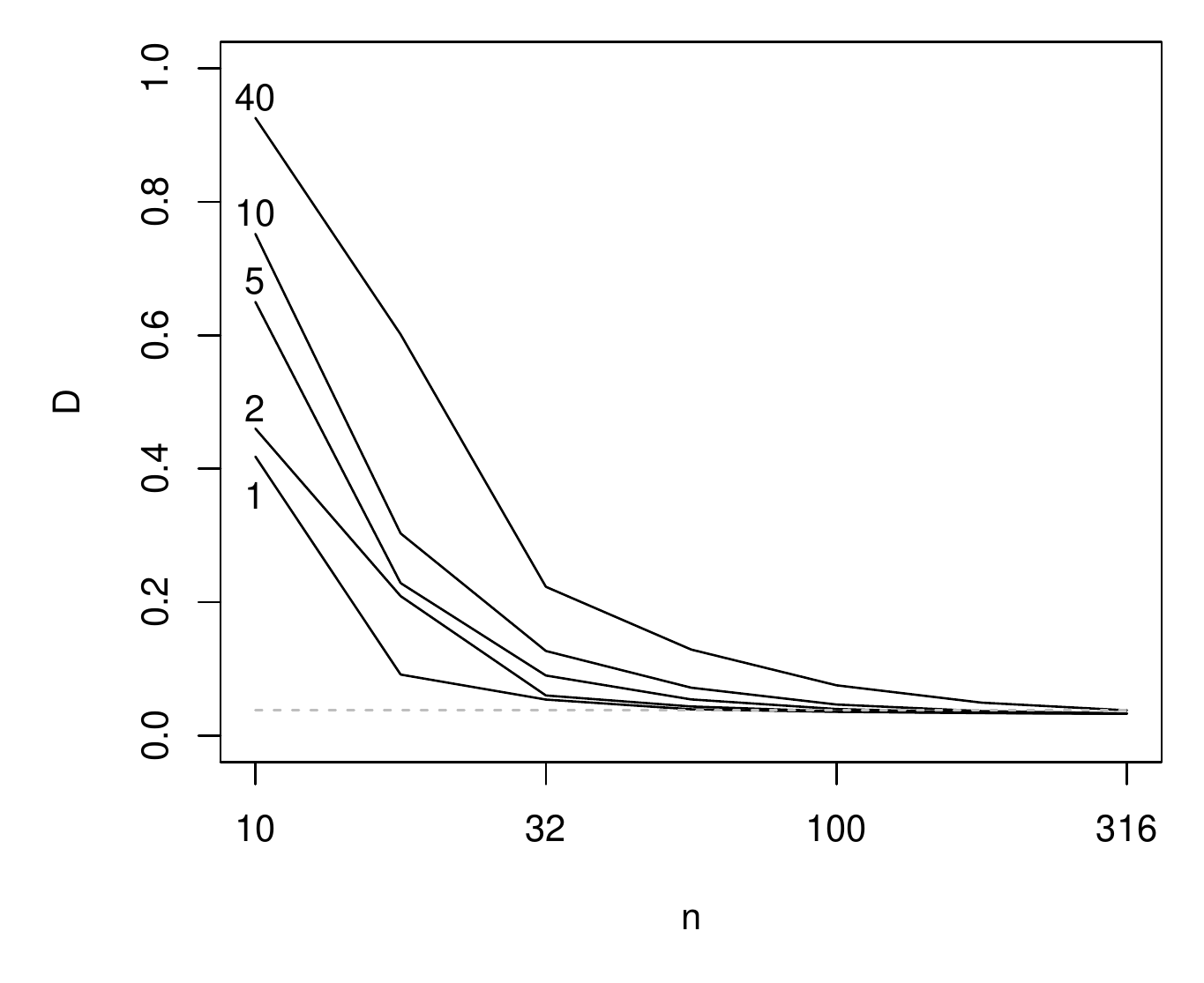}
\caption{$D$}
\label{fig:ses}
\end{subfigure}
\caption{Randomization Metrics by $n$ and $p \in \braces{1,2,5,10,40}$. The dotted line plots the metric for completely randomized vectors. Both metrics were estimated with $r=1,000$ runs of the algorithm from different starting points and different realizations of the normal covariates. The level of $p$ is labeled for each line in the plot.}
\label{fig:rand}
\end{figure}

\section{Inference from a Greedy Pair-Switching Design}\label{sec:inference}

Consider an experiment where the randomization is allocated via our greedy pair switching procedure. Responses for the treatment group $Y_{T,1}, \ldots, Y_{T,n}$ and for the control group $Y_{C,1}, \ldots, Y_{C,n}$ are collected and we wish to make inference for the additive treatment effect $\beta$ in the model $Y_{T,i} = \beta + f(X_{T,i}) + \errorrv_{T,i}$ and $Y_{C,i} = f(X_{C,i}) + \errorrv_{C,i}$ where $X$ denotes the appropriate length covariate vector, $f$ denotes the not-necessarily linear conditional expectation function of the response and the $\errorrv$'s denote mean-centered noise independent of the $X$'s.

How can we draw frequentist inference for $\beta$ to generate a confidence interval or test a hypothesis? One can use the classic unbiased \emph{differences-in-means estimator}, $\hat{\beta} := \overline{Y_T} - \overline{Y_C}$. We first prove consistency of this estimator when the allocation vector was designed using our greedy pair-switching algorithm under any response model $f$ (even if the treatment affect is not additive).

\subsection{Consistency of the Differences-in-Means Estimator}\label{sec:consistency}

Let $\bar{Y}^0_C$, $\bar{Y}^0_T$ to be the mean responses of the control and treatment groups, respectively, before the first iteration and $\bar{Y}^{f}_C$, $\bar{Y}^{f}_T$ are the final mean responses of the greedy algorithm. 
We have that

\beqn
(\bar{Y}^0_T - \bar{Y}^0_C)-(\bar{Y}^{f}_T - \bar{Y}^{f}_C) ~=~ \frac{2}{n}\sum_{\substack{\text{only the} \\ \text{switched} \\ \text{pairs}~(i,j)}}(Y_i - Y_j) ~\leq~ \frac{2}{n} N_s (\Ymax-\Ymin),
\eeqn

where $\Ymax := \max\{Y_1,\ldots, Y_{2n}\}$, $\Ymin := \min\{Y_1,\ldots,Y_{2n}\}$ and $N_s$ is the number of switches. We would like to show that


\bneqn\label{eq:convp}
(\bar{Y}^0_T - \bar{Y}^0_C) - (\bar{Y}^{f}_T - \bar{Y}^{f}_C) \convp 0
\eneqn

which means that $\bar{Y}^{f}_T - \bar{Y}^{f}_C$ is consistent whenever $\bar{Y}^0_T - \bar{Y}^0_C$ is consistent (the latter holding true under any model). All that is missing is demonstrating that if $Y$ has a finite second moment then $\max\{Y_1,\ldots, Y_{2n}\}-\min\{Y_1,\ldots, Y_{2n}\}=o_p(\sqrt{n})$. We do so in the following proposition. \\ 

\begin{proposition}\label{prop:one}
Consider $Y$ a random response unconditioned on the assignment. Suppose that $\expe{Y^2} < \infty$, then $\frac{\max\{Y_1,\ldots,Y_{2n}\}}{\sqrt{n}} \convp 0$ and $-\frac{\min\{Y_1,\ldots,Y_{2n}\}}{\sqrt{n}} \convp 0$.
\end{proposition}

\begin{proof}
We only show that the maximum is $o_p(\sqrt{n})$ since the equivalent argument for the minimum is symmetric. For notational convenience, let $M_n=\max\{Y_1, \ldots, Y_{n}\}$.

We first show that $\frac{M_n}{\sqrt{n}} \convp 0$ if $\frac{1-\prob{Y \le t}}{1/t^2} \tendt 0$. Fix $\varepsilon,t_0 >0$, by the condition, $\prob{Y \le t} \ge 1-\varepsilon/t^2$ for large enough $t$. Hence, for large enough $n$, 
\[
\prob{\frac{M_n}{\sqrt{n}} \le t_0} = \prob{{M_n} \le t_0 \sqrt{n}} = \prob{Y \le t_0 \sqrt{n}}^n \ge (1-\varepsilon/(t_0^2n))^n \tendn \exp{-\varepsilon/t_0^2}.
\]
Since this is true for all $\varepsilon > 0$ then, $\prob{\frac{M_n}{\sqrt{n}} \le t_0} \tendn 1$ for every $t_0>0$, i.e., $\frac{M_n}{\sqrt{n}} \convp 0$. 

We now show that if $Y$ has a second moment then the condition $\frac{1-\prob{Y \le t}}{1/t^2} \tendt 0$ is satisfied. We have that 
\[
\expe{Y^2 \indic{Y>0}} = \int_0^\infty \prob{Y^2 \indic{Y>0}  > t} dt = \int_0^\infty \prob{Y > \sqrt{t}} dt.
\]
Since this integral converges, then the integrand goes to zero faster than $1/t$, i.e., $ \frac{\prob{Y > \sqrt{t}}}{1/t} \tendt 0$ or equivalently that $\frac{\prob{Y > {t}}}{1/t^2} \tendt 0$. 
\end{proof}

\subsection{Valid Testing and Confidence Interval Construction}

Classically, the scaled differences-in-means estimate is compared to known quantiles of Student's $T$ distribution. However, this estimator ignores the contribution of $X$ and thus will not have a $T$ distribution. For example, in the case where $f$ is linear in $X$, the differences-in-means estimator will have a spurious noncentrality parameter in the numerator of the $T$ statistic \citep[page 409]{Efron1971}. 

Further, we used $X$ to create the allocation which improves the balance $B$, \qu{typically [creating a] more preceise estimated treatment [effect], making traditional Gaussian distribution-based forms of analysis statistically too conservative} \citep[page 6]{Morgan2012}. A more stern warning  is found in \citet[Section 2.1]{Senn2000} who writes that allocating a randomization using covariate information \qu{that are [later] not included in the model} is a decision which is \qu{pointless if not harmful} and \qu{incoherent}. Simulation results illustrating a closely-related concept can be found in \citet[Table 3]{Kapelner2014}.

We can also see from the reasoning in Section~\ref{sec:consistency} that $\bar{Y}^{f}_C - \bar{Y}^{f}_T$ does not have the same asymptotic distribution as $\bar{Y}^0_C-\bar{Y}^0_T$ since $\sqrt{n}((\bar{Y}^0_C-\bar{Y}^0_T) - (\bar{Y}^f_C-\bar{Y}^f_T))$ does not converge to 0 unlike when the difference is not multiplied by $\sqrt{n}$ in (\ref{eq:convp}).  Indeed, recent follow-up work on the difference-in-means estimator after allocation by \citeauthor{Morgan2012}'s (\citeyear{Morgan2012}) rerandomization procedure shows that their resulting estimator is asymptotically non-Gaussian \citep{Li2016}. It is possible our procedure can be shown to fit the same criterion proved therein but we leave this to future work.

One can find an estimator with asymptotic normality in the example in \citet[page 409]{Efron1971} by using a multivariable linear regression and reporting the slope coefficient of the binary allocation vector (assuming an assignment of 1 indicates the subject was assigned the treatment), a common practice in applied work from social science to medicine. \citet{Freedman2008} proves that this estimator has the problems of being biased itself and having biased standard errors. Thus, it may perform worse than the differences-in-means estimator and \qu{the reason for the breakdown is not hard to find: randomization does not justify the assumptions behind the OLS model} and we have no reason to rely on these assumptions in practice.

What procedure can we use to attain valid inference? We follow \citet[Section 2.2]{Morgan2012}, quoting Fisher, Tukey and others, who recommend the permutation test. \qu{This test can incorporate whatever rerandomization procedure was used, will preserve the significance level of the test and works for any estimator}. A permutation test can assess the validity of any \qu{sharp} null hypothesis (e.g. in a clinical trial implementation of our design, the sharp null hypothesis would be that each subject will have the same response under both treatment and control). Note that we are not required to understand the asymptotic distribution of the differences-in-means estimator nor its standard error using the strategy. The procedure is as follows.

We first select a resolution $R$ for inference permutation-based inference. We then run the greedy pair-switching algorithm for $R + 1$ replicates. We choose one allocation at random to provide to the experimenter as the true randomization of the subjects. The experiment is performed, the $Y$'s are collected and the differences-in-means estimate $\hat{\beta}$ is computed. For each of the remaining $R$ allocations, a faux differences-in-means estimate $\hat{\beta}_r$ is computed by imagining the replicate allocation vector was the true design. Since the relationship between assignment and response is now broken, the $\hat{\beta}_1, \ldots, \hat{\beta}_R$ constitute an empirical null distribution that the treatment has no effect on any subject. An approximate $\alpha$-level two-sided test for instance can be performed by creating a rejection region smaller than the $\alpha/2$ quantile and larger than the $1-\alpha/2$ quantile. Further, the set of all values not rejected under the sharp null of the treatment effect being $\hat{\beta}$ at significance level $\alpha$ would, by the duality of confidence intervals and hypothesis tests, constitute an approximate $1-\alpha$ confidence interval. 

\section{Concluding Remarks}\label{sec:conclusion}

We introduce a new randomization design procedure called greedy pair-switching which provides better balance while keeping designs close to what would be expected under complete randomization.

In addition, we considered covariates that are absolutely continuous. In practice, it is likely that some of the covariates are categorical. This is not an issue theoretically. For example, if two of the covariates of interest are gender and education, treated as a categorical variable with four levels then there are eight (and in general a finite number of) subpopulations based on the categorical variables. We can apply the greedy heuristic to each of these subpopulations. We can then consider each subpopulation in turn and use the greedy heuristic. This implies that all categorical variables will be balanced and all continuous variables will provide a $B=O(n^{-2/p})$.  However, the practical aspects of the result apply to relative size values of $n$, say $n=100$. The notion of considering subpopulations might not work that well in this case as dividing $n$ into say eight groups renders at least some small subpopulations. Alternatively, we can apply the greedy heuristic to the continuous variables, but only consider switches that do not increase the balance in any of the categorical variables.

One virtue of using the greedy algorithm is that one can apply it to any objective. As such, we can apply the greedy heuristic to the Mahalanobis distance directly (and our results will be of the same order as shown in Section~\ref{subsec:choice}) or to other popular balance metrics \citep[see][Figure 2]{Franklin2014}. There are additional possiblities for customization as well. Consider the situation where one felt that some covariates are more important than others. Here, the switches can be made such that a differentially weighted average of the balances across the $p$ covariates is minimized.

Also, since computation is relatively inexpensive, the greedy pair-switching algorithm can be repeated $r$ times and the minimum balance vector can be cherry picked. This further improves the order of $B$ by a factor of $1/r$.

\subsection*{Replication}

All figures and tables can be reproduced by running the \texttt{R} code found at \url{https://github.com/kapelner/GreedyExperimentalDesign/blob/master/paper_duplication.R}.

\subsection*{Acknowledgements}

We thank Elchanan Mossel and Adi Wyner for their helpful comments and Ishay Weissman for help with the proof of Proposition~\ref{prop:one}.

\bibliographystyle{apa}
\bibliography{refs}

\end{document}